\documentclass[11pt]{amsart}
\usepackage{amsmath,amsfonts,amsthm,amssymb,bbm,tikz}
\usetikzlibrary{decorations.pathmorphing}
\usepackage[margin=1.5in]{geometry}

\newcommand{\reals}{\mathbb{R}}

\newcommand{\Var}[1]{\mathbf{Var}(#1)}

\newcommand{\Dform}[2]{\mathcal{Q}\left(#1,#2 \right)}
\newcommand{\tripNorm}[1]{|||#1|||_\infty}

\newtheorem{thm}{Theorem}[section]

\newtheorem{lem}{Lemma}[section]

\theoremstyle{definition}

\begin{document}
\title{A method to derive concentration of measure bounds on Markov chains}
\author[S. Ng]{Stephen Ng}
\address{Department of Mathematics\\
University of Rochester\\
Rochester, NY 14627, USA}
\email{ng@math.rochester.edu}

\author[M. Walters]{Meg Walters}
\address{Department of Mathematics\\
University of Rochester\\
Rochester, NY 14627, USA}
\email{walters@math.rochester.edu}

\date{\today}

\begin{abstract}
We explore a method introduced by Chatterjee and Ledoux in a paper on eigenvalues of principle submatrices.  The method provides a tool to prove concentration of measure in cases where there is a Markov chain meeting certain conditions, and where the spectral gap of the chain is known.  We provide several additional applications of this method.  These applications  include results on operator compressions using the Kac walk on $SO(n)$ and a Kac walk coupled to a thermostat, and a concentration of measure result for the length of the longest increasing subsequence of a random walk distributed under the invariant measure for the asymmetric exclusion process.   
\end{abstract}
\maketitle
\section{Introduction}
In the analysis of Chatterjee and Ledoux on concentration of measure for random submatrices \cite{ChatLed} , it is proved that for an arbitrary Hermitian matrix of order $n$ and $k\leq n$ sufficiently large, the distribution of eigenvalues is almost the same for any principal submatrix of order $k$.  Their proof uses the random transposition walk on $S_n$ and concentration of measure techniques.  To further generalize their results, we observe that it is important to use a Markov chain which does not change too many matrix entries all at once and whose spectral gap is known. To demonstrate that this method can be generalized to a much wider range of problems, we provide three applications.  As our first application, instead of looking at a Markov chain on $S_n$, we first consider a Markov chain on $SO(n)$.  We introduce the Kac walk on $SO(n)$ and demonstrate that it is sufficiently similar to the transposition Markov chain to allow for Chatterjee and Ledoux's results to carry over to the more general case of operator compressions.  It should be noted that a similar result has been proved by Meckes and Meckes \cite {MeckesMeckes} using different techniques.  In a more recent work \cite{MeckesMeckes2}, Meckes and Meckes have extended their techniques to include several other classes of random matrices and prove almost sure convergence of the empirical spectral measure.   As the purpose of this paper is to highlight the fact that the methods of Chatterjee and Ledoux can be extended to include more general cases, we include this operator compression result, as it a straightforward application and serves as a useful example for us to explain the method in detail.  As a second application, we apply the method to get a concentration of measure result for a compression by a matrix of Gaussians using the Kac walk coupled to a thermostat. We also show that the method can be applied to get concentration of measure of the length of the longest increasing subsequence of a random walk evolving under the asymmetric exclusion process.  This method opens the door to concentration of measure in settings where one has an appropriate underlying Markov process with a known spectral gap.  

\section{Overview of method}
Before diving into the applications, we would like to give a brief overview of the method.  We will then show how to calculate concentration of measure in our applications using this technique.  To use the method, we must start with a stationary, reversible Markov chain for which the spectral gap is known.  We denote the Markov chain by $X_0, X_1, \dots$.  Call the state space of the Markov chain $\mathcal{S}$.  We will denote the invariant distribution as $\pi$ and the spectral gap as $\lambda_1$.  For a function $f:\mathcal{S}\rightarrow\mathbb{R}$, define 
$$
\||f\||_{\infty}^2 :=\frac{1}{2}\sup_{x\in S}\mathbb{E}((g(X_1)-g(X_0))^2 | X_0=x)
$$
and 
$$
\mathcal{Q}(f,f):=\frac{1}{2}\mathbb{E}((f(X_1)-f(X_0))^2)
$$
The Poincare inequality tells us that 
$$
\mathcal{Q}(f,f) \geq \lambda_1 \mathrm{Var}(f(X_0))
$$
In order for the method to work properly, $\||f\||^2_{\infty}$ must be bounded.  An important step in all of our applications will be finding a bound for $f$, so for now, assume that $\||f\||^2_{\infty}<\delta$.  

We begin by applying the Poincare inequality to $e^{tf(X_0)}$ for $t\geq 0$.  This gives 
$$
\lambda_1\mathrm{Var}(e^{tf(X_0)}) \leq \mathcal{Q}(e^{tf(X_0)}, e^{tf(X_0)})
$$
$$
=\frac{1}{2}\mathbb{E}(e^{tf(X_1)}-e^{tf(X_0)})^2
$$
$$
 \mathbb{E}\big(\mathbbm{1}_{f(X_0)\geq f(X_1)}(e^{tf(X_1)}-e^{tf(X_0)})^2\big)
$$
$$
=\mathbb{E}(\mathbb{E}(\mathbbm{1}_{f(X_0)\geq f(X_1)}(e^{t(f(X_1)-f(X_0))}-1)^2|X_0)e^{2tf(X_0)})
$$
$$
\leq t^2\mathbb{E}(\mathbb{E}(\mathbbm{1}_{f(X_0)\geq f(X_1)}(f(X_0)-f(X_1))^2|X_0)e^{2tf(X_0)})
$$
$$
\leq t^2\||f\||^2_{\infty}\mathbb{E}(e^{2tf(X_0)})
$$
We then define $\Lambda(t):=e^{-t\mathbb{E}f(X_0)}\mathbb{E}(e^{tf(X_0)})$ and use recursion to show that $\Lambda(c\sqrt{\lambda_1/\delta})\leq C<\infty$ for explicit values of $c$ and $C$.  Chebyshev's inequality then leads to 
$$
\mathbb{P}(f(X_0)\geq\mathbb{E}(f(X))+r)\leq Ce^{-cr\sqrt{\lambda_1/\delta}}
$$
for $r>0$.  Once we have this, the method can be applied after choosing and appropriate Markov chain and finding $\lambda_1$ and $\delta$.  Further details will be provided in the applications.  

\section{The Kac walk on $SO(n)$}
The following model, introduced by Kac \cite{Kac}, describes a system of particles evolving under a random collision mechanism such that the total energy of the system is conserved.  Given a system of $n$ particles in one dimension, the state of the system is specified by $\vec{v}=(v_1,\dots v_n)$, the velocities of the particles.  At a time step $t$, $i$ and $j$ are chosen uniformly at random from $\{1,\dots, n\}$ and $\theta$ is chosen uniformly at random on $(-\pi, \pi]$.  The $i$ and $j$ correspond to a collision between particles $i$ and $j$ such that the energy,
$$
E=\sum_{k=1}^n v_k^2
$$
is conserved.  Under this constraint, after a collision, the new velocities will be of the form $v_i^{\mathrm{new}}=v_i\cos(\theta)+v_j\sin(\theta)$ and $v_j^{\mathrm{new}}=v_j\cos(\theta)-v_i\sin(\theta)$.    
For $i<j$, let $R_{ij}(\theta)$ be the rotation matrix given by:
\begin{equation*}
	R_{ij} (\theta) = \begin{pmatrix} I & & & & \\
		&\cos(\theta) &&\sin(\theta)& \\
		&&I&& \\
		&-\sin(\theta) &&\cos(\theta)& \\
		&&&& I \end{pmatrix}
\end{equation*}
where the $\cos(\theta)$ and $\sin(\theta)$ terms are in rows and columns labeled $i$ and $j$, and the $I$ denote identity matrices of different sizes (possibly 0). We will use the convention that $R_{ii}{\theta}=I$.  After one step of the process, $\vec{v}_{new}=R_{ij}(\theta)\vec{v}$.  

In our case, we will be considering this process acting on $SO(n)$, so instead of vectors in $\mathbb{R}^n$, our states will be given by matrices $G\in SO(n)$.  Then we can define the one-step Markov transition operator for the Kac walk, $Q$, on continuous functions of $SO(n)$:
\begin{equation} \label{Kac}
	Qf(G) = \frac{1}{\binom{n}{2}} \sum_{i<j} \int_0^{2\pi} f(R_{ij}(\theta)G)\frac{1}{2\pi} d\theta
\end{equation}
for any $G\in SO(n)$, and where $f$ is a continuous function on $SO(n)$. 

\begin{thm}[\cite{CarlenCarvalhoLoss,Maslen}]
	The Kac walk on $SO(n)$ is ergodic and its invariant distribution is the uniform distribution on $SO(n)$. Furthermore, the spectral gap of the Kac walk on $SO(n)$ is $\frac{n+2}{2(n-1)n}$.
\end{thm}

Recall that for any reversible Markov chain, we can define the Dirichlet form, $\Dform{\cdot}{\cdot}$. It is well known that for a Markov chain with spectral gap, $\lambda_1$, the Poincare inequality holds:
\begin{equation*}
	\lambda_1 \Var{f} \leq \Dform{f}{f}.
\end{equation*}
For the Kac walk, we have 
\begin{equation*}
	\Dform{f}{f} = \frac 1 {2\binom{n}{2}} \sum_{1 \leq i<j\leq n} \int_0^{2\pi} \frac{1}{2\pi} \int_{SO(n)} \left(f(G) - f(R_{ij}(\theta)G)\right)^2 d\mu_n(G) d\theta,
\end{equation*}
where $\mu_n$ is the Haar measure on $SO(n)$ normalized so that the total measure is $1$.

Let us define the triple norm:
\begin{equation}
	\tripNorm{f}^2 = \frac 1 {2\binom{n}{2}} \sup_{G\in SO(n)}\sum_{1 \leq i<j\leq n} \int_0^{2\pi} \frac{1}{2\pi} \left|f(G) - f(R_{ij}(\theta)G)\right|^2 d\theta.
\end{equation}
The following result is analogous to Theorem 3.3 from Ledoux's Concentration of Measure Phenomenon book \cite{LedouxCOM} . We reproduce the proof of Theorem 3.3 here to verify that even though our situation does not satisfy the conditions of the theorem, the exact same argument carries through for the Kac walk on $SO(n)$. 
\begin{thm}
	Consider the Kac walk on $SO(n)$ and let $F:SO(n)\to \reals$ be given such that $\tripNorm{F}\leq 1$. Then $F$ is integrable with respect to $\mu_n$ and for every $r\geq 0$, 
	\begin{equation*}
		\mu_n(F \geq \int F d\mu_n + r) \leq 3 e^{-r\sqrt{\lambda_1}/2}
	\end{equation*}
	where $\lambda_1= \frac{n+2}{2(n-1)n}$ is the spectral gap of the Kac walk on $SO(n)$.
\end{thm}
\begin{proof}
	We first demonstrate that $\Dform{e^{\lambda F/2}}{e^{\lambda F/2}})\leq \frac{\lambda^2 \tripNorm{F}^2}{4} \int_{SO(n)} e^{\lambda F(G)} d\mu_n(G)$ by using symmetry.
	\begin{align*}
		\Dform{e^{\lambda F/2}}{e^{\lambda F/2}} &= \frac 1 {2\binom{n}{2}} \sum_{1 \leq i<j\leq n} \int_0^{2\pi} \frac{1}{2\pi} \int_{SO(n)} \left(e^{\lambda F(G)/2} - e^{\lambda F(R_{ij}(\theta)G)/2}\right)^2 d\mu_n(G) d\theta \\
		& =  \frac 1 {\binom{n}{2}} \sum_{1 \leq i<j\leq n} \int_0^{2\pi} \frac{1}{2\pi} \int_{F(G)>F(R_{ij}(\theta)G)} \left(e^{\lambda F(G)/2} - e^{\lambda F(R_{ij}(\theta)G)/2}\right)^2 d\mu_n(G) d\theta \\
		& \leq \frac{\lambda^2}{4} \frac 1 {\binom{n}{2}} \sum_{1 \leq i<j\leq n} \int_0^{2\pi} \frac{1}{2\pi} \int_{SO(n)} \left(F(G) - F(R_{ij}(\theta)G)\right)^2 e^{\lambda F(G)} d\mu_n(G) d\theta \\
		& = \frac{\lambda^2}{4} \tripNorm{F}^2 \int_{SO(n)} e^{\lambda F(G)} d\mu_n(G)
	\end{align*}
	Setting $\Lambda(\lambda) =e^{-\lambda \int_{SO(n)}F(G)d\mu_n(G) }\int_{SO(n)} e^{\lambda F(G)} d\mu_n(G)$, we combine this with the Poincare inequality to obtain
	\begin{equation*}
		\lambda_1 \Var{e^{\lambda F/2}} = \lambda_1 \left(\Lambda(\lambda)-\Lambda^2\left(\frac{\lambda}2\right)\right) \leq \Dform{e^{\lambda F/2}}{e^{\lambda F/2}} \leq \frac{\lambda^2}{4} \tripNorm{F}^2 \Lambda(\lambda).
	\end{equation*}
	Incorporating the assumption $\tripNorm{F}\leq1$ yields
	\begin{equation*}
		\Lambda(\lambda) \leq \frac{1}{1-\frac{\lambda^2}{4\lambda_1}} \Lambda^2(\lambda/2).
	\end{equation*}
	Iterating the inequality $n$ times gives
	\begin{equation*}
		\Lambda(\lambda) \leq \prod_{k=0}^{n-1} \left(\frac{1}{1-\frac{\lambda^2}{4^{k+1}\lambda_1}}\right)^{2^k} \Lambda^{2^n}(\lambda/2^n).
	\end{equation*}
	Since $\Lambda(\lambda) = 1+ o(\lambda)$, we see that $\Lambda^{2^n}(\lambda/2^n) \to 1$ as $n\to \infty$. This gives the upper bound
	\begin{equation*}
		\Lambda(\lambda)\leq \prod_{k=0}^\infty \left(\frac{1}{1-\frac{\lambda^2}{4^{k+1}\lambda_1}}\right)^{2^k}. 
	\end{equation*}
	By plugging in $\lambda=\sqrt{\lambda_1}$,using the crude estimate $\prod_{k=0}^\infty \left(\frac{1}{1-\frac{1}{4^{k+1}}}\right)^{2^k} < 3$, and applying Chebyshev's inequality, we obtain the result. 
\end{proof}

\section{First Application: Random Operator Compressions}
 Following the notation of Chatterjee and Ledoux, for a given Hermitian matrix $A$ of order $n$ with eigenvalues given by $\lambda_1,\dots,\lambda_n$, we let $F_A$ denote the empirical distribution function of $A$.  This is defined as
$$
F_A(x):=\frac{\#\{i:\lambda_i\leq x\}}{n}
$$
Using the results from above, along with the method of Chatterjee and Ledoux, we are able to prove the following result:
\begin{thm}  Take any $1\leq k\leq n$ and an $n$-dimensional Hermitian matrix $G$.  Let $A$ be the $k\times k$ matrix consisting of the first $k$ rows and $k$ columns of the matrix obtained by conjugating $G$ by a rotation matrix $R^{\theta}_{ij}\in SO(n)$ chosen uniformly at random.  If we let $F$ be the expected spectral distribution of $A$, then for each $r>0$, 
\begin{equation*}
\mathbb{P}(\|F_A-F\|_{\infty}\geq \frac{1}{\sqrt{k}}+r)\leq 12\sqrt{k}\exp\left(-r\sqrt{\frac{k}{32}}\right)
\end{equation*}
\end{thm}
\begin{proof}
The proof of this theorem uses the method introduced by Chatterjee and Ledoux \cite {ChatLed} with appropriate changes made to apply to the situation we are considering.  \\ \\ 
Let $R_{ij}(\theta)\in\mathrm{SO}(n)$ and let $A$ be as stated above.  Note that since $A$ is a compression of a Hermitian operator, it will also be Hermitian.  Fix $x\in\mathbb{R}$.  Let $f(A):=F_A(x)$, where $F_A(x)$ is the empirical spectral distribution of $A$.  Let $Q$ be the transition operator as defined in (1) and let $\tripNorm{.}$ be as in (2).  Using Lemma 2.2 from Bai\cite{Bai}, we know that for any two Hermitian matrices $A$ and $B$ of order $k$, 
\begin{equation*} \|F_A-F_B\|_{\infty}\leq\frac{\mathrm{rank}(A-B)}{k}
\end{equation*}
In our case, taking one step in the Kac walk is equivalent to rotation in a random plane by a random angle.  Hence $A$ and $R_{ij}^{\theta}A$ will differ in at most two rows and two columns, bounding the difference in rank by $2$, so
\begin{equation*}
\|f(A)-f(R_{ij}^{\theta}A)\|_{\infty}\leq\frac{2}{k}
\end{equation*}
Using (2), 
\begin{equation*}
\tripNorm{f}^2 = \frac 1 {2\binom{n}{2}} \sup_{A\in SO(n)}\sum_{1\leq i<j\leq n} \mathbb{E}[f(A)-f(R_{ij}^{\theta}A)]^2
\end{equation*}
\begin{equation*}
\leq \frac 1 {2} \left(\frac{2}{k}\right)^2\left(\frac{2k}{n}\right)=\frac{4}{kn}
\end{equation*} where the $\frac{2k}{n}$ comes from the probability that both $i$ and $j$ are greater than $k$, in which case, $A$ and $R^{\theta}_{ij}A$ will be the same.
From Theorems 2.1 and 2.2, we have that 
\begin{equation*}
\mathbb{P}(|F_A(x)-F(x)|\geq r)\leq 6\exp\left(-\frac{r}{2}\frac{\sqrt{\frac{1}{2}\frac{n+2}{(n-1)n}}}{\sqrt{\frac{4}{kn}}}\right)
\end{equation*}
\begin{equation*}
=6\exp\left(-r/2\sqrt{\frac{1}{8}\frac{k(n+2)}{n-1}}\right)\leq 6\exp\left(-r/2\sqrt{\frac{k}{8}}\right)
\end{equation*}
This is true for any $x$.  Now, if we let $F_A(x-):=\lim_{y\uparrow x}F_A(y)$, then we have $\mathbb{E}F_A(x-)=\lim_{y\uparrow x}F(y)=F(x-)$.  Hence, for $r>0$, 
\begin{equation*}
\mathbb{P}(|F_A(x-)-\mathbb{E}F_A(x-)|>r)\leq \lim_{y\uparrow x}\mathbb{P}(|F_A(y)-F(y)|>r)
\end{equation*}
\begin{equation*}
\leq 6\exp\left(-r/2\sqrt{\frac{k}{8}}\right)
\end{equation*}
The steps to get from $\mathbb{P}(|F_A(x)-F(x)|)$ to $\mathbb{P}(\|F_A-F\|_{\infty})$ are identical to the steps in the original Chatterjee and Ledoux paper, so we will omit them here.  After completing these steps, we are left with
\begin{equation*}
\mathbb{P}(\|F_A-F\|_{\infty}\geq \frac{1}{\sqrt{k}}+r)\leq 12\sqrt{k}\exp\left(-r\sqrt{\frac{k}{32}}\right)
\end{equation*} which concludes the proof of our theorem.  
\end{proof}
 
\section{Second Application: Kac Model Coupled to a Thermostat}
Using a spectral gap result from \cite{BonLossVaid}, we are able to demonstrate the application of this method to a more complicated Markov chain.  In this system, the particles from the Kac system interact amongst themselves with a rate $\lambda$ and interact with a particle from a thermostat  with rate $\mu$.  The particles in the thermostat are Gaussian with variance $\frac{1}{\beta}$, so they have already reached equilibrium.  If we let $f_t(\bf{v})$ denote the probability distribution of finding the system at time $t$ with velocities $\bf{v}$, then the master equation for the Kac model coupled to a thermostat is given by 
$$
\frac{\partial f}{\partial t} = -\lambda N(I-Q)[f]-\mu\sum_{j=1}^N (1-R_j)[f]
$$
where $N$ denotes the number of particles in the system, $Q$ is the Markov transition operator for Kac walk (as seen in equation \ref{Kac}), and 
\begin{equation}
R_if(G)=\frac{1}{n}\sum_{j=1}^n \frac{1}{2\pi}\int_{0}^{2\pi}\int_{\mathbb{R}^n}
 \sqrt{\frac{\beta}{2\pi}}^ne^{-\frac{\beta}{2}\omega_{ij}^{*2}(\theta)}f(V_j(\theta, \omega)G)d\theta d\omega
\end{equation}
where $\omega=(\omega_1,\omega_2,\dots,\omega_n)$, $V_j(\theta, \omega)$ sends each element $g_{ij}$ in column $j$ to $g_{ij}cos(\theta)+\omega_i\sin(\theta)$ for $i=1$ to $n$ and $\omega_{ij}^*=-g_{ij}\sin(\theta)+\omega_i\cos(\theta)$. 
In \cite{BonLossVaid} they consider the Markov chain acting on a vector.  We consider the Markov chain acting on a matrix by treating the matrix as $n$ independent vectors.  Using this adaption, the following theorem follows immediately from the results proved in \cite{BonLossVaid}.  
\begin{thm}  The Kac walk coupled to a thermostat is ergodic and has unique invariant measure given by $$\nu_n=\prod_{i,j} \sqrt{\frac{\beta}{2\pi}}e^{-\frac{\beta}{2}v_{ij}^2}$$ and has spectral gap $\frac{\mu}{2n}$  \end{thm} 
For the thermostat alone (letting $\lambda=0$), we can again prove a theorem analogous to Chatterjee and Ledoux's theorem 3.3.  Let $\mathcal{G}$ be the set of $n\times n$ matrices with independent and identically distributed $N(0,1/\beta)$ entries.  We can define the Dirichlet form and the triple norm for the thermostat as 
\begin{equation*}
\mathcal{Q}(f,f)=\frac{1}{2n}\sum_{j=1}^n\frac{1}{2\pi}\int_0^{2\pi}\int_{\mathbb{R}^n}\int_{G\in\mathcal{G}}\left(\frac{\beta}{2\pi}\right)^{n/2}e^{-\frac{\beta}{2}w_{ij}^{*2}}(f(V_j(\theta,w))G-f(G))d\nu_ndwd\theta
\end{equation*}
\begin{equation}
\tripNorm{f}^2=\sup_{G\in\mathcal{G}}\; \;\frac{1}{2n}\sum_{j=1}^n\frac{1}{2\pi}\int_0^{2\pi}\int_{\mathbb{R}^n}\left(\frac{\beta}{2\pi}\right)^{n/2}e^{-\frac{\beta}{2}w_j^{*2}}|f(V_j(\theta,w))G-f(G)|^2dwd\theta
\end{equation}
Using these, we can prove a concentration of measure result for the thermostat analogous to Theorem 2.2
\begin{thm} Consider the Gaussian thermostat and let $F:\mathcal{G}\rightarrow\mathbb{R}$ be such that $\tripNorm{F}\leq 1$.  Then $F$ is integrable with respect to $\nu_n$ and for every $r\geq 0$, 
$$
\nu_n(F\geq Fd\nu_n+r)\leq 3e^{-r\sqrt{\lambda_1}/2}
$$
where $\lambda_1=\frac{\mu}{2n}$ is the spectral gap of the thermostat process.  
\end{thm}
We omit the proof here as it is symmetric to the proof of Theorem 2.2.  \\ \\
Using this result and Theorem 4.1, we can prove the following concentration of measure inequality.  
\begin{thm}
Take any $1\leq k\leq n$ and an $n$-dimensional Hermitian matrix $G$.  Let $S$ be an $n\times k$ matrix whose $k$ columns are the first $k$ columns of a random matrix with distribution $\nu_n$.  Let $A$ be the $k\times k$ matrix obtained by conjugating $G$ by $S$.  Letting $F$ denote the expected spectral distribution of $A$, then for each $r>0$, 
$$
\mathbb{P}(\|F_A-F\|_{\infty}\geq \frac{1}{\sqrt{k}}+r)\leq 12\sqrt{k}\exp\left(-r\sqrt{\frac{k\mu}{108}}\right)
$$ where $\mu$ is the rate of the interaction with the thermostat.  
\end{thm}

\begin{proof}
The proof of this theorem closely follows the proof of Theorem 3.1, with appropriate changes made.  Let $A$ be stated as above, and let $A'$ be $A$ after one step of the Markov chain.  Fix $x\in\mathbb{R}$ and let $f(x)=F_A(x)$, where where $F_A$ is the empirical spectral distribution of $A$.  Notice that rank($A-A')\leq 3$, since after one step of the chain, at most 3 columns of $A$ will be changed (two from the Kac Walk, and one from the thermostat).  Again using the inequality from \cite{Bai}, we know that 
$$
\|f(A)-f(A')\|_{\infty}\leq\frac{3}{k}
$$
$$
\tripNorm{f}^2 =\frac{1}{2{n\choose 2}{n}}\sup_A\sum_{1\leq i <j\leq n} \sum_{k=1}^n\mathbb{E}|f(A)-f(A')|^2
$$
where the first sum is over possible interactions in the Kac process and the second is over possible particle interactions with the thermostat.  The above is 
$$
\leq \frac{1}{2}\left(\frac{3}{k}\right)^2\left(\frac{3k}{n}\right)=\frac{27}{2kn}
$$
Using theorems 4.1 and 4.2, we have that 
$$
\mathbb{P}(|F_A(x)-F(x)|\geq r)\leq 6\exp\left(-\frac{r}{2}\sqrt{\frac{\frac{\mu}{2n}}{\frac{27}{2kn}}}\right)
$$
$$
=6\exp\left(-\frac{r}{2}\sqrt{\frac{k\mu}{27}}\right)
$$
Following the rest of the proof in 2.1 (with the appropriate numbers changed), we get 
$$
\mathbb{P}(\|F_A-F\|_{\infty}\geq \frac{1}{\sqrt{k}}+r)\leq 12\sqrt{k}\exp\left(-r\sqrt{\frac{k\mu}{108}}\right)
$$
\end{proof}

\section{Third Application: The Length of the Longest Increasing Subsequence of a Random Walk Evolving under the Asymmetric Exclusion Process}. 

Consider a random walk X on $\{1,\dots,n\}$.  Represent $X$ by some element in $\{0,1\}^n$, where  $X_i=0$ corresponds to a step down in the walk at position $i$ and $X_i=1$ corresponds to a step up.  We will assume that 
$$
\sum_{i=1}^n X_i =\frac{n}{2}
$$
so that we have the same number of up steps as down steps. 
We can now look at this random walk as the initial configuration of a particle process with $X_i=1$ corresponding to a particle in position $i$ and $X_i=0$ corresponding to no particle at position $i$.  Consider the asymmetric exclusion process acting on this configuration with the following dynamics.  At each step of the process, a number $i$ is chosen uniformly in $\{1,\dots,n-1\}$.  If $X_i=X_{i+1}$, then the configuration stays the same.  If $X_i=1$ and $X_{i+1}=0$, then the values of $X_i$ and $X_{i+1}$ switch with probability $1-q/2$ and if $X_i=0$ and $X_{i+1}=1$, then the values switch with probability $q/2$.  Viewed in this way, the asymmetric exclusion process can be viewed as a Markov process on the set of random walks.  See \cite{Liggett} for an in depth discussion of the asymmetric exclusion process.

\begin{thm}[\cite{KomaNach},\cite{Alcaraz},\cite{CapMar}] The spectral gap of the ASEP is $\lambda_n=1-\Delta^{-1}\cos(\pi/n)$, where $\Delta=\frac{q+q^{-1}}{2}$ for a parameter $q$ satisfying $0<q<1$.
\end{thm}
In our case, take $q=1-c/n^{\alpha}$, for a constant $c$, and $0<\alpha<1$, such that 
$q\approx e^{-c/n^{\alpha}}$.  Then Taylor approximating and simplifying gives 
$$
\lambda_n=c^2/2n^{2\alpha}
$$
Now let $M_X$ denote the height of the midpoint of the random walk at a fixed time during the process.  In other words, $M_X=X_{n/2}$, assuming $n$ is even.  Note that the range of this function is $[-n/2,n/2]$. Let $M_x'$ be the evolution of $M_x$ after one step of the process.  Notice that $$\|M_x-M_x'\|_{\infty}\leq 1$$ since switching the position of two adjacent particles can change the height of the midpoint by at most $1$.  Then 
$$
\||M||_{\infty}^2=\frac{1}{2}\max_{X}\mathbb{E}(M_x-M_x')^2
$$
$$
\leq \frac{1}{2}(1)^2\left(\frac{1}{n-1}\right)=\frac{1}{2(n-1)}
$$
The $\frac{1}{n-1}$ appears because the only choice of $i$ that will effect the midpoint is $i=n/2$.  \\ \\
Now plugging into the Chatterjee Ledoux theorem, we have the following result.
\begin{thm}
Letting $M_X$ denote the height of the midpoint of the random walk after evolution under the asymmetric exclusion process, for all $r>0$ and $q=1-c/n^{\alpha}$,
$$
\mathbb{P}(|M_X-\mathbb{E}M_X|\geq r)\leq 6\exp\left(-r/2\sqrt{\frac{c^2/2n^{2\alpha}}{1/(2(n-1))}}\right) =6\exp\left(-r/2\sqrt{\frac{c^2(n-1)}{n^{2\alpha}}}\right)
$$
\end{thm}
\noindent
Notice that this implies that the height of the midpoint has fluctuations bounded above by a constant $n^{\alpha-1/2}$  for $0<\alpha<1$. 

Consider the length of the longest increasing (non-decreasing) subsequence of the random walk.  This is defined as $$ L_X=\max\{k \;   : \; i_1<i_2<\dots<i_k \; \mathrm{and} \; X_{i_1}\leq X_{i_2}\leq\dots\leq X_{i_k}\}$$  See \cite{PeresRW} for a more in depth description of this topic and results for the simple random walk.  

Notice that the height of the midpoint gives a lower bound on the length of the longest increasing subsequence.  Using ASEP as our Markov process and the spectral gap above, we can prove concentration of measure for $L_X$.  Notice that switching the position of two adjacent particles via ASEP can only change $L_X$ by at most $1$.  As before, let $X'$ be the evolution of $X$ after one step of the process.  Then, bounding the probability above by $1$, we have 
$$
\||L\||_{\infty}^2=\frac{1}{2}\max_X\mathbb{E}(L_X-L_{X'})^2
$$
$$
\leq \frac{1}{2} (1)^2=\frac{1}{2}
$$
so plugging into the Chatterjee Ledoux formula, we get the following result.
\begin{thm}
Letting $L_X$ denote the length of the longest increasing subsequence of the random walk after evolution under the asymmetric exclusion process, for all $r>0$ and $q=1-c/n^{\alpha}$, 
$$
\mathbb{P}(|L_X-\mathbb{E}L_X|\geq r)\leq 6\exp\left(-r/2\sqrt{\frac{c^2}{n^{2\alpha}}}\right)
$$
\end{thm}
\noindent
This implies that the fluctuations are bounded above by a constant times $n^{\alpha}$.  In particular, for $q=1-c/\sqrt{n}$, the fluctuations are bounded above by a constant times $\sqrt{n}$. 
 
In order to give some context to the size of the fluctuations, we calculate height of the midpoint, which gives a lower bound on the length of the longest increasing subsequence of the walk under this distribution.
\begin{thm}
\label{height}
For $q<1-c/n$ and $c=-20\log(3/5)$, the height of the midpoint of the random walk is $kn$ for some constant $k>0$.  
\end{thm}

Before we give the proof, we will need the following lemma.
\begin{lem}
Consider a random walk with independent steps.  Assume that $\mathbb{P}(X_k=0)=\frac{1}{aq^k+1}$ and $\mathbb{P}(X_k=1)=\frac{aq^k}{aq^k+1}$ for some $a>0$, $q\in (0,1)$ and $k\in\mathbb{Z}_{+}$.  Consider $N_X=\sum_{i=1}^n X_i$.  This gives us the number of up steps in our random walk, or equivalently, the number of particles in our particle process.  The fluctuations of $N_X$ are at most order $\sqrt{n}$.
\end{lem}
\begin{proof}
We begin by calculating the variance of $N_X$.  We can then use Chebyshev's inequality to bound the fluctuations.  Since the $X_i$ are independent, 
$$
\mathrm{Var}(N_X)=\sum_{i=1}^n \mathrm{Var}(X_i)
$$
Using the probabilities given in the lemma, we know that 
$$
\mathrm{Var}(X_i)=\frac{aq^i}{aq^i+1}-\left (\frac{aq^i}{aq^i+1}\right)^2
$$
$$
=\frac{aq^i}{aq^i+1}\left (1-\frac{aq^i}{aq^i+1}\right)
$$
This gives 
$$
\mathrm{Var}(N_X)=\sum_{i=1}^n\frac{aq^i}{aq^i+1}\left (1-\frac{aq^i}{aq^i+1}\right)
$$
A derivative calculation show that $\frac{aq^i}{aq^i+1}\left (1-\frac{aq^i}{aq^i+1}\right)$ is decreasing in $i$, so 
$$
\mathrm{Var}(N_X)\leq n \left(\frac{aq}{aq+1}\right)\left(1-\frac{aq}{aq+1}\right)
$$
Since we only care about the order of the fluctuations, we can bound the positive value
$$
\left(\frac{aq}{aq+1}\right)\left(1-\frac{aq}{aq+1}\right)
$$
by $1$, giving us 
$$
\mathrm{Var}(N_X)\leq n
$$
Plugging into Chebyshev's inequality tells us that 
$$\mathbb{P}\left(|N_X-\mathbb{E}(N_X)|\geq k\right)\leq \frac{n}{k^2}$$
which proves our result.  

\end{proof}
We are now set to prove theorem \ref{height}
\begin{proof}
The basic idea of the proof of theorem \ref{height} is as follows.  We will begin by assuming that the steps of our random walk are independent, so that our measure is a product measure.  Recall, the steps are not independent, since we are conditioning on the fact that we have exactly $n/2$ steps up and $n/2$ steps down.  However, if $n$ is large, the steps are {\em close} to independent.  By bounding the fluctuations of the number of particles in our product system, we can then relate our non-independent state to the product state. 

Begin by assuming that 
$$
\frac{P(X_k=0)}{P(X_k=1)}=aq^k
$$
so that we have a product measure. 
Then we know that 
$$
P(X_k=0)=\frac{1}{aq^k+1}
$$ and $$P(X_k=1)=\frac{aq^k}{aq^k+1}$$
Then 
$$
\mathbb{E}\left(\sum_{i=1}^k X_i\right)=\sum_{i=1}^k\frac{aq^i}{aq^i+1}$$
Since the summand is decreasing in 
$i$, we get the bounds 
$$
k\left(\frac{aq^k}{aq^k+1}\right)\leq \mathbb{E}\left(\sum_{i=1}^k  X_i\right)\leq k\left(\frac{aq}{aq+1}\right)
$$
We will work in this generality for now, and add in appropriate values of $a$ and $k$ later.  Using this information, we can get bounds on the height of the random walk at point $k$.  Let $H_k$ be the height of the random walk at position $k$. For convenience later, we will assume that $X_i=1$ corresponds to a step down in the walk, and that $X_i=0$ corresponds to a step up.  Provided that we can prove that our height is $cn$ for $c<0$, our theorem will be proved.  We have  
$$
\mathbb{E}(H_k)=(-1)\sum_{i=1}^k X_i+\left(k-\sum_{i=1}^k X_i\right)=k-2\left(\sum_{i=1}^k X_i\right)
$$
Plugging in our bounds on $\mathbb{E}\left(\sum_{i=1}^k X_i\right)$, we get 
$$
-k\left(2\left(\frac{aq}{aq+1}\right)-1\right)\leq\mathbb{E}(H_k)\leq -k\left(2\left(\frac{aq^k}{aq^k+1}
\right)-1\right)
$$

At this point, we need a bound on the number of particles in the system.  
Since we are assuming the $X_i$ are independent, we can use the result from the previous lemma, which gives us
$$
\mathbb{P}\left(\left |\sum_{i=1}^n X_i-M\right |>u\right)\leq 4\exp(-u^2/4M)
$$ 
where $M$ is a median for the number of particles.  Estimating the median by the expectation of the number of particles, we see that $M$ should at least be close to $n/2\left(\frac{aq}{aq+1}\right)$.  If we choose $a$ appropriately corresponding to $q$, we should be able to make the constant order $1$, making our expectation order $n$.  Then, by the concentration of measure inequality, $\sum_{i=1}^n X_i$  has  fluctuations on the order of $\sqrt{n}$.  This is reasonably small compared with the expected number of particles in the system. 

Recall that we are actually concerned with finding the height of the midpoint, so plugging in $k=n/2$, we have that 
$$
-n/2\left(2\left(\frac{aq}{aq+1}\right)-1\right)\leq\mathbb{E}(H_{n/2})\leq -n/2\left(2\left(\frac{aq^{n/2}}{aq^{n/2}+1}
\right)-1\right)
$$
At this point, we can ignore the lower bound, using the fact that that a lower bound is $-n/2$ anyway, regardless of the configuration.  We will refer to our interface as the position in which $\mathbb{P}(X=0)=\mathbb{P}(X=1)$.  For now, we will put our interface at $9n/20$, which will be just to the left of the midpoint.   In other words, $a=q^{-9n/20}$ and at position $9n/20$, $\mathbb{P}(X=0)=\mathbb{P}(X=1)$. We will push it to the edge at $n/2$ at the end, since moving the interface to the right only increases the probability of more $X_i$ being equal to $1$, hence lowering the expectation of the midpoint.  Using this interface, we will first look at the height of the random walk at position $8n/20$.  Using the upper bound from above, we have that 
$$
\mathbb{E}(H_{8n/20})\leq \frac{-8n}{20}\left(2\left(\frac{q^{-n/20}}{q^{-n/20}+1}\right)-1\right)
$$
Beyond this point, if we assume that all of the remaining steps between $8n/20$ and $n/2$ are steps up, we have that 
$$
\mathbb{E}(H_{n/2})\leq \frac{-8n}{20}\left(2\left(\frac{q^{-n/20}}{q^{-n/20}+1}\right)-1\right) + \frac{2n}{20}
$$
The important thing to notice here, is this actually gives us an upper bound on the height of the midpoint in the fixed particle number (ASEP) random walk.  In the product state configuration, with our interface at $\frac{9n}{20}$, we know that the fluctuations in the number of down steps are less than $\frac{n}{20}$.  By assuming that all steps after site $\frac{8n}{20}$ are up, we have accounted for the worst case scenario where we actually have $\sqrt{n}$ less down steps then we expect.  If some of the steps after site $\frac{8n}{20}$ are actually down instead of up, this will only serve to lower the height of our midpoint.  Hence, we have, that in the ASEP (fixed number of down steps) random walk generated using the blocking measures, 
$$
\mathbb{E}(H_{n/2})\leq \mathbb{E}(H_{n/2})\leq \frac{-8n}{20}\left(2\left(\frac{q^{-n/20}}{q^{-n/20}+1}\right)-1\right) + \frac{2n}{20}
$$
We would like to show that for an appropriate choice of $q$, this is $cn$ for some constant $c<0$.  This is true provided that 
$$
\frac{8}{20}\left(2\left(\frac{q^{-n/20}}{q^{-n/20}+1}\right)-1\right)>\frac{2}{20}
$$
Solving this inequality gives a condition on q, which is 
$$
q>\left(\frac{3}{5}\right)^{\frac{20}{n}}
$$
or 
$$
q>e^{20/n\log(3/5)}
$$
Taylor expanding the exponential gives 
$$
q> 1+\frac{20}{n}\log(3/5)+\frac{400}{2n^2}(\log(3/5))^2+\dots
$$
As $n\rightarrow\infty$, taking $q>1-\alpha/n$ with $\alpha=-20\log(3/5)$ should be sufficient.  
As long as this condition is satisfied, our expectation is $cn$ for a constant $c<0$.  

At this point, we do want to move the interface to $a=q^{-n/2}$, such that $\mathbb{P}(X_{n/2}=0)=\mathbb{P}(X_{n/2}=1)$.  This simply increases our probability of down steps between $\frac{9n}{20}$ and $\frac{n}{2}$.  Since adding extra down steps only decreases the expectation of the height of the midpoint, the theorem is proved. 
\end{proof}
  
\section{Remarks}
Using this method, we are able to show concentration of measure of the empirical spectral distribution not only for operator compressions via $SO(n)$ but also for operators that are "compressed" by conjugation with a Gaussian matrix.  We are also able to use the method to prove a concentration of measure result for the length of the longest increasing subsequence of a random walk.  It is likely that this method could be applied to a much wider range of Markov chains, given that the chain does not change too many entries at once, has an appropriate invariant distribution, and for which the spectral gap is known.  It is possible that better bounds for the Gaussian compression could be obtained by adapting the method to use the "second" spectral gap or the exponential decay rate in relative entropy found in \cite{BonLossVaid}. 

It is worth noting that Talagrand's isoperimetric inequality \cite{Talagrand} gives concentration of measure for the length of the longest increasing subsequence for random permutations, but it cannot be used in the context of this ASEP random walk, as it requires independence.  Using Chatterjee and Ledoux's method, independence is not needed.  We only need a spectral gap bound for the Markov chain.    
\\ \\
{\bf Acknowledgements: } We would like to thank Shannon Starr for suggesting this problem to us and for many helpful discussions and comments.   

\bibliographystyle{plain}

\end{document}